\documentclass{birkjour}
 \newtheorem{thm}{Theorem}[section]
 \newtheorem{cor}[thm]{Corollary}
 \newtheorem{lem}[thm]{Lemma}
 \newtheorem{prop}[thm]{Proposition}
 \theoremstyle{definition}
 
 \theoremstyle{remark}
 \newtheorem{rem}[thm]{Remark}
 
 \numberwithin{equation}{section}

\begin{document}

%
%
%
%
%
%
%
%
%

\title[Generalised Poisson Transform of the hyperbolic space $B({\mathbb{H}}^{n})$]
 {Characterisation of the ${L}^{p}$ Range of the Generalised Poisson Transform of the hyperbolic space $B({\mathbb{H}}^{n})$}

\author[Imane Ghanimi]{Imane Ghanimi.}

\address{Department of Mathematics\br
Faculty of sciences\br
University Ibn Tofail\br
Kenitra, Morocco}
\email{ghanimiimane@gmail.com}

\thanks{}



\date{September 19, 2017}
\dedicatory{In memory of Professor Ahmed Intissar}

\begin{abstract}
The aim of this paper is to give the characterisation of the ${L}^{p}$ Range $(p\geq 2)$  of the Generalised Poisson Transform of the Hyperbolic space $B({\mathbb{H}}^{n}),(n\geq 2) $, over the classical field of the quaternions $\mathbb{H}$.
Namely, if $f$ is an hyperfunction in the boundary of $B({\mathbb{H}}^{n})$, then we show that $f$ is in ${L}^{p}(\partial B({\mathbb{H}}^{n}))$  if and only if it's generalised poisson transform satisfy an  Hardy type growth condition. An explicit expression of the generalized spherical functions is given.
\end{abstract}

\maketitle

\section{Introduction }

Riemannian symmetric spaces of non compact type $ G/K $  are of great importance for various branches of mathematics. Morever they give an interesting connection between the theory of group representations and Special functions. The investigation of this relation gave a powerful push of the development of physics.\\

Harmonic analysis on non compact riemannian symmetric spaces was first developped in the 1950's with harish chandra. It's has been next intesively studied by many people, namely by Helgason and it is nowaday well inderstood. It becomes then natural to extend the theory to homogeneous vector bundle over   $ G/K $.\\

Many autors examined this theory and tried to formulate a generalization of the whole harmonic analysis over $ G/K $, while other people have intensively studied some special contexts of this theory: main tools are the theory of spherical functions, Fourier and Poisson transform.\\

Let $ X $ denote a non compact Riemannian symmetric space of rank one. Then $ X $ can  be identified with $ G/K $ where $ G $ is a connected noncompact semi simple Lie group with finite center having real rank one, and $ K $ is a maximal compact subgroup of $ G $. An irreductible unitary representation $\tau $ of $ K $ is canonically associated to a vector bundle over $ E_{\tau} $ over $G/K  $.\\

In this work, we restrict ourselves to the case where $ X $ is the quaternionic hyperbolic space. This framework was previously treated in [1],  the autors gived a complet treatment of the spherical transform on vector bundle over $ X $. The aim of this paper is to give a characterisation of the generalized poisson transform related to an irreductible representation $\tau_{l} $ of $ K $ using $\tau_{l} $-spherical functions, and $\tau_{l} $-generalized spherical functions.\\
 
Let's now describe in more details the background of this paper. As hyperbolic space, it is known that the boundary of $ X$ is representation of degenerate principal series; the generalized poisson transform is an intetwiting operator from representation of degenerate principal series to the space of "eigensections" of the vector bundle over $G/K $. The spherical functions arises then as a generalized poisson transfrom of a particular function in the boundary of $ X $. Following the techniques of Takahashi in [3] we give the expression of these spherical functions as hypergeometric function, which allows as to give the proof of the necessary condtion of our main theorem.\\

Next we recall the decomposition due to Peter-Weyl $ {L}^{2}(\partial B(\mathbb{H}^{n}))= \sum_{(p,q) \in {V}_{p,q}}{H}_{p,q} $, the zonal spherical harmonics that spans $ {H}_{p,q} $ are given in [5]; then we define an operator $ P_{\lambda,l,p,q} $ from each $ {H}_{p,q} $ to ${L}^{2}(\partial B(\mathbb{H}^{n}))  $, and by Schur's lemma we conclude that the operator $ P_{\lambda,l,p,q}  $ is scalar on each composant $ {H}_{p,q} $. This leads to the existence of the generalized spherical functions: Poisson transform of the zonal spherical harmonics.\\

Using some computational methods for special functions, we are able to give explictly the expression of these generalized spherical functions. Indeed, they arise as a linear combination of two contiguous hypergeometric functions. The main difficulty in the proof of our main theorem, after the explicit computation of the generalized spherical functions, is to find their asymptotic behaviour. This problem has been resolved in [2]. This allows as to prove the sufficient condition of our main theorem for the case $ \textbf{p}=2 $. The case $ \textbf{p}> 2 $ is given by an inversion formula.\\

\section{Notations, preliminaries and statement of the main results}

Let $n\geq 2$ be an integer, and let $\mathbb{H}$ be the classical field of the quaternions. ${\mathbb{H}}^{n+1} $ being considered as left vector space, let's define the following Lorentz form $ L $ on $ {\mathbb{H}}^{n+1} $:\\

$ L(x,y)= \overline{y_{1}}x_{1}+ \overline{y_{2}}x_{2}+...-\overline{y_{n+1}}x_{n+1}$ where $ \overline{y_{i}} $ is the quaternionic conjugate of $ \overline{y_{i}} $.

 Let's consider $B({\mathbb{H}}^{n})$ be the bounded realisation of the Hyperbolic space $X := G/K$, where:\\

$G = Sp(n,1)$ is the subgroup of $ GL(n+1,{\mathbb{H}}) $ preserving $ L $, and \\
$K =\lbrace k = \left(
                                \begin{array}{cc}
                                  A & 0 \\
                                  0 & D \\
                                \end{array}
                              \right)
;A \in Sp(n), D \in Sp(1)\rbrace$ \\

Recall that then $ G $ is a connected non compact semi-simple real Lie group with finite center and $ K $ is a maximal compact subgroup of $ G $. The quaternionic hyperbolic space $X := G/K$ can be realized as a non compact Riemannian space of rank one and of real dimension $ 4n $.\\

But we know that $Sp(1)$  is canonically isomorphic to $SU(2)$, then the set of equivalence classes of unitary irreductible representations of $Sp(1)$ is parametrized by the set of non negative integers  $\mathbb{N}/2$  ;\\

Lets denotes by ${\tau}_{l}$ the equivalence classe corresponding to the parameter $l$ and let define $\tau(k):= \tau_{l}(D)$ for every $k$ in $K$;\\

Thus, $\tau_{l}$ is a unitary irreductible representation of $Sp(1)$ in a hilbert space ${V}_{l}$ of dimension $dim{V}_{l}= 2l+1$, and $\tau_{l}$ can be extended to a representation of $K$ by setting $\tau_{l}\equiv 1$ on $Sp(n)$;\\

 Furthermore, each $\tau_{l}$ is self-dual, it follows that the character of $\tau_{l}$, ${\chi}_{l}=tr({ \tau_{l}})$ satisfies : ${\chi}_{l}({k}^{-1})={\chi}_{l}({k})$ for every $k$ in $K$;\\
 
 And we also have: ${\chi}_{l}(q/|q|)={C}^{1}_{2l}(Re(q/|q|)= sin(2l+1)\theta/sin\theta
$, where $q:=r(cos\theta+ysin\theta) \in \mathbb{H}$\\

$\tau_{l}$ is canonically related to an homogeneous vector bundle $ E^{\tau} = G \times_{K}V_{l} $
and the space of sections of $ E^{\tau} $ can be identified with the space:\\
$ \Gamma(G, \tau) = \lbrace F: G \rightarrow V_{l}, F(gk) = \tau(k)^{-1}F(g), \forall g\in G, \forall k\in K \rbrace$\\

 let's denote:\\
  $ \mathcal{C}^{\infty}(G, \tau)= \Gamma(G, \tau)\cap \mathcal{C}^{\infty}(G, V_{l})$ , $ \mathcal{C}^{\infty}(G, V_{l}) $ being the space of $ \mathcal{C}^{\infty} $ functions $ F: G \rightarrow V_{l} $.\\

Recall that we have: $ L^{2}(G, \tau)\simeq \lbrace L^{2}(G)\bigotimes V_{l}\rbrace ^{K} $ where the upper index  $ K $ means that we take the subspace of $ K $-invariant vectors for the right action of $ K $ on $ L^{2}(G) $ and the induced bundle  $ L^{2}(G, \tau) $ may sit in $  L^{2}(G) $ in a naturel way.[4][6]\\

Let $ \mathfrak{a} $ be a maximal abelian subspace in $ \mathfrak{p} $ and let $ \mathfrak{m}\simeq sp(n-1)\bigoplus sp(1) $ be the centralizer of $ \mathfrak{a} $ in $ \mathfrak{k} $ then we have the Iwasawa decomposition :\\
$ \mathfrak{g}= \mathfrak{k}\bigoplus \mathfrak{a} \bigoplus \mathfrak{n} $ and we also have the identification: $ \mathfrak{a}^{\star} = \mathbb{R} $, and let $ \mathfrak{g} = sp(n,1) $ and $ \mathfrak{k}= sp(n)\bigoplus sp(1) $ be the Lie algebras of $ G $ and let $  \mathfrak{g} = \mathfrak{k}\bigoplus \mathfrak{p}$ be the Cartan decomposition.\\

Let's denote by $\rho$ the half sum of positive roots of the pair $(\mathfrak{g},\mathfrak{a})$. For instance, $\rho = 2n+1$;\\

let $ P:= MAN $ be the the standard minimal parabolic subgroup of $ G $ where $ A $, $ M $, $ N $ are the analytics subgroups related respectively to $ \mathfrak{a} $, $ \mathfrak{m} $ and $ \mathfrak{n} $, and for $ t \in \mathbb{R} $ set $ A = \lbrace a_{t}, t \in \mathbb{R}\rbrace$; For each $  g = k a_{t} n \in G $  , we'll denote $ \kappa(g)= k $ and $ t=  t(g) $
 \\

Let's denote by  $ \partial B({\mathbb{H}}^{n})= G/P \simeq K/M $ the boundary of $ G/K $\\

Now we are ready to define a generalized Poisson transform referring the reader to [7],[9] for more general framework.\\

For all $ \lambda \in \mathbb{C} $ define the representation $ \sigma_{\lambda,l} $ of $ P = MAN $ by:\\

  $ \sigma_{\lambda,l}(m a_{t} n) = e^{-(\lambda-\rho)t}\tau_{l}(m) $\\
 
Now let's consider the representation $ \pi_{\lambda,l} = Ind^{G}_{P}(\sigma_{\lambda,l}) $ the representation of $ G $ induced by the character $ \sigma_{\lambda,l} $; The representation $ \pi_{\lambda,l} $ is realized in $\mathcal{H}_{\lambda,l}  $ ,the Hilbert completion  of the space:\\

$ \mathcal{H}_{\lambda,l}^{\infty} = \lbrace F: G\rightarrow V_{l}, \mathcal{C}^{\infty}, F(gma_{t}n) = e^{(\lambda-\rho)t}\tau_{l}(m^{-1})  F(g),\\
 \forall g \in G, t \in \mathbb{R}, n\in N \rbrace $ \\

$\mathcal{H}_{\lambda,l}^{\infty}  $ may be identified with the space:

  $ \mathcal{C}^{\infty}(G/P,l) = \lbrace f  : G\rightarrow \mathbb{C},  \mathcal{C}^{\infty}, f(g)= 1/d_{l}\int_{M}f(gm)\chi_{l}(m)dm,\\
   f(ga_{t}n) = e^{(\lambda - \rho)t} f(g), \ \forall g \in G, t \in \mathbb{R}, n\in N\rbrace $ \\
  
 Then $\pi_{\lambda,l} $ acts on the Hilbert space :\\
 
  $ \mathcal{H}_{\lambda,l} = \lbrace F: G\rightarrow V_{l},
  \mathcal{C}^{\infty},   F(gma_{t}n) = e^{(\lambda-\rho)t}\tau_{l}(m^{-1})  F(g), \forall g \in G, t \in \mathbb{R}, n\in N, F/_{K} \in L^{2}(K)  \rbrace $ by left translations.\\
   
As $ K $-module, $\mathcal{H}_{\lambda,l}  $ is identified with  $ L^{2}(K, \sigma)$ : the space of $ L^{2} $ functions on $ K $ of right type $ \sigma := \tau_{l}/M $; more precisely, the restriction map $f \mapsto f/_{K } $ induces a bijective isometry from $\mathcal{H}_{\lambda,l}  $ to the space $ L^{2}(K, \sigma)$.\\

Morever, $\pi_{\lambda,l} $ acts on $ L^{2}(K, \sigma)$ by the rule: $ \pi_{\lambda,l}(f)(k)= e^{(\lambda-\rho)t(g^{-1}k})f(\kappa(g^{-1}k)) $ \\

For $ g \in G $ put:    $ \mathtt{P}_{\lambda,l}(g):= e^{-(\lambda+\rho)t(g)}\chi_{l}(\mathit{\kappa}(g))  $ \\

For $ f \in \mathcal{C}^{\infty}(G/P,l) $ the generalized Poisson transform is defined as follows:
  $ P_{\lambda,l}f(g):= \int_{K} \mathtt{P}_{\lambda,l}(k^{-1}gk)f(k)dk $\\
    
which gives:

     $ P_{\lambda,l}f(g)= \int_{K} e^{-(\lambda+\rho)t(g^{-1}k)}\chi_{l}(\kappa(g^{-1}k)k^{-1})f(k)dk $\\
     
Denote by $ \mathcal{A}'(\partial B(\mathbb{H}^{n})$ the space of hyperfunctions on the boundary $ \partial B(\mathbb{H}^{n})$ of $B(\mathbb{H}^{n})$.\\
 
For $\lambda\in \mathbb{C}$ a straightforward computation shows that the generalised Poisson transform  may be realized in   $\mathcal{A}'(\partial B(\mathbb{H}^{n}))$ as follows:

\begin{equation*}
{P}_{\lambda,l}(f)(x)= \int_{\partial B({\mathbb{H}}^{n})}(\frac{1- {|x|}^{2}  }{{|1-<x,\omega> |}^{2}})^{{\frac{i\lambda + 2n+1}{2}}}  {\chi}_{l}{(\frac{1-<x,\omega> }{|1-< x,\omega > |})}f(\omega)d \omega
\end{equation*}

Now we may state the main theorem of this paper:\\

\begin{thm} 
Let $f$ be in $\mathcal{A}'(\partial B(\mathbb{H}^{n})$, then we have for $Re(i\lambda)>0$  and $\textbf{p} \geq 2 $ :\\
\begin{equation*}
 \|P_{\lambda,l}f\|_{\lambda,\ast,\textbf{p}} < +\infty \Leftrightarrow f \in {L}^{\textbf{p}}(\partial B(\mathbb{H}^{n}))
\end{equation*}

Where:
\begin{equation*}
\|P_{\lambda,l}f\|_{\lambda,\ast,\textbf{p}}=sup _{0\leq r <1}{(1-{r}^{2})}^{\frac{(2n+1-Re(i\lambda)}{2}}{[\int_{\partial B(\mathbb{H}^{n})}{|(P_{\lambda,l}f)(r\omega)|}^{\textbf{p}} d\omega]}^{1/\textbf{p}}.
\end{equation*}
Morever, there exists to positive constants $C_{l}(\lambda)$ and $\delta_{l}(\lambda)$ such as:
\begin{equation*}
C_{l}(\lambda){\|f\|}_{\textbf{p}} \leq \|P_{\lambda,l}f\|_{\lambda,\ast,\textbf{p}}  \leq \delta_{l}(\lambda) {\|f\|}_{\textbf{p}}
\end{equation*}

 \end{thm}

To obtain our main result, we first prove the case $\textbf{p}=2$. Namely we prove the following theorem:

\begin{thm} \label{some label} 
Let $f$ be in $\mathcal{A}'(\partial B(\mathbb{H}^{n})$, then we have for $Re(i\lambda)>0$  :\\
\begin{equation*}
 \|P_{\lambda,l}f\|_{\lambda,\ast,2} < +\infty \Leftrightarrow f \in {L}^{2}(\partial B(\mathbb{H}^{n}))
\end{equation*}

Where:
\begin{equation*}
\|P_{\lambda,l}f\|_{\lambda,\ast,2}=sup _{0\leq r <1}{(1-{r}^{2})}^{\frac{(2n+1-Re(i\lambda)}{2}}{[\int_{\partial B(\mathbb{H}^{n})}{|(P_{\lambda,l}f)(r\omega)|}^{2} d\omega]}^{1/2}.
\end{equation*}

\end{thm}

The main theorem follows from the inversion formula given in the following theorem:

\begin{thm} \label{some label} 
Let $F = P_{\lambda,l}f, f \in {L}^{2}(\partial B(\mathbb{H}^{n}))  $. Then its ${L}^{2}$- boundary value is given by the following inversion formula:

\begin{equation*}
 f(u)= {(C_{l}(\lambda))}^{-2} lim_{r\mapsto {1}^{-}} {(1-{r}^{2})}^{-\frac{(2n+1-Re(i\lambda))}{2}}\int_{\partial B(\mathbb{H}^{n})}F(rv)  \overline{{P}_{\lambda,l}(ru,v)}dv 
\end{equation*}
in ${L}^{2}(\partial B(\mathbb{H}^{n}))$
Where $C_{l}(\lambda)$ is given by Corollary(5.3)( see section 5).
\end{thm}

 This paper is organised as follows:
 
 In Section3 we determine the elementary spherical function ; 
 
 In Section4 we describe the precise action  of the generalized Poisson transform on ${L}^{2}(\partial B(\mathbb{H}^{n}))$ and define the generalized spherical functions;
 
 In Section5 we determine explicitely the generalized spherical functions and conclude their asymptotic behaviour;

And in the last section we turn to the proof of the main theorem.
 
\section{Expression of the elementary spherical function }

Before giving the main result of this section let's first state the following proposition which will be useful in the sequel:
\begin{prop}
The genenralized poisson kernel given by: 
\begin{equation*}
{P}_{\lambda,l}(x,\omega)= (\frac{1- {|x|}^{2}  }{{|1-<x,\omega> |}^{2}})^{{\frac{i\lambda + \rho}{2}}}  {\chi}_{l}{(\frac{1-<x,\omega> }{|1-< x,\omega > |})}; x,\omega \in \partial B(\mathbb{H}^{n})
\end{equation*} is $K$ invariant. 
\end{prop}

\begin{proof}

Let $k = \left(
                                \begin{array}{cc}
                                  A & 0 \\
                                  0 & D \\
                                \end{array}
                              \right)$ be in $K$, then we have for every $x$ and $\omega$ in $ \partial B(\mathbb{H}^{n}) $:

   \begin{equation*}
   < kx,k\omega > = < Ax\overline{D},A\omega\overline{D} > = ^{t}D. ^{t}\overline{x}(^{t}\overline{A}.A)\omega\overline{D}
   \end{equation*} 
   But $A \in Sp(n)$ then $^{t}\overline{A}.A = Id$ and:
 \begin{equation*}                         
< kx,k\omega > = ^{t}D. ^{t}\overline{x}\omega\overline{D}= ^{t}D <x,\omega> \overline{D}
 \end{equation*}
 Then we have:
 \begin{equation*}
{P}_{\lambda,l}(kx,k\omega)= (\frac{1- {|x|}^{2}  }{{|1-^{t}D <x,\omega> \overline{D}|}^{2}})^{{\frac{i\lambda + \rho}{2}}}  {\chi}_{l}{(\frac{1-^{t}D<x,\omega>\overline{D} }{|1-^{t}D< x,\omega >\overline{D} |})}
\end{equation*}

 And finally, recalling that ${\chi}_{l}$ is a character we conclude that:
 
 \begin{equation*}{\chi}_{l}{(\frac{D(1-<x,\omega>)\overline{D} }{|1-< x,\omega >|})}={\chi}_{l}{(\frac{1-<x,\omega> }{|1-< x,\omega > |})}
  \end{equation*}
  
  which gives:
   \begin{equation*}
   {P}_{\lambda,l}(kx,k\omega)={P}_{\lambda,l}(x,\omega)
 \end{equation*}
\end{proof}

Now let's recall the following lemma (see for instance[8]):
\begin{lem}
Let $f$ be a $\mathbb{C}$ valued function on $\partial B({\mathbb{H}}^{n})$ with:
\begin{equation*}
 f(\omega)= g({\omega}_{1}), \omega = ({\omega}_{1},...,{\omega}_{n}) \in \partial B({\mathbb{H}}^{n})
\end{equation*}
Then we have for some positive constant $C$ :
\begin{equation*}
{\int}_{\partial B(\mathbb{H}^{n})}f({\omega})d\omega= C  {\int}_{\{\textbf{q}\in \mathbb{H}, |\textbf{q}|<1\}}  f(\textbf{q}){(1-{|\textbf{q}|}^{2})}^{\rho-d}dm(\textbf{q})
\end{equation*}
\end{lem}

Let's denote by $ e_{1} $ the unit vector of $ {\mathbb{H}}^{n} $ and define the elementary spherical function as being the mean of the generalised Poisson transform at $tht.{e}_{1}, t \in\mathbb{R} $:\\

 ${\Phi}_{\lambda,l}(t):= ({P}_{\lambda,l}1)(tht.{e}_{1})$\\

Then we may state the following proposition wich will be useful in the proof of the main theorem:

\begin{prop}
The elementary spherical function $ {\Phi}_{\lambda,l}(t)$ is given by:\\
 $ \dfrac{\pi}{4}(2l+1) \frac{\Gamma(2)\Gamma(2n-2)}{\Gamma(2n)}{(1-{tht}^{2})}^{\frac{i\lambda+2n+1}{2}}{F}_{21}(\frac{i\lambda+2n+1}{2}+l,\frac{i\lambda+2n+1}{2}-l-1;2n;{tht}^{2}) $

\end{prop}

\begin{proof}

\begin{equation*}
 {\Phi}_{\lambda,l}(t)= {\int}_{\partial B({\mathbb{H}}^{n})}{(\frac{1-{tht}^{2}}{{|1-tht{\omega}_{1}|}^{2}})}^{\frac{i\lambda+\rho}{2}}{\chi}_{l}(\frac{1-tht{\omega}_{1}}{|1-tht{\omega}_{1}|})d{\omega}_{1} ;
\end{equation*}

Using the lemma below and setting $s:=\frac{i\lambda+\rho}{2}$ we can write :

 \begin{equation*}
{ \Phi}_{\lambda,l}(t)= {(1-{tht}^{2})}^{s} {\int}_{\{\textbf{q}\in \mathbb{H}, |\textbf{q}|<1\}}{|1-tht\textbf{q}|}^{-2s}{\chi}_{l}(\frac{1-tht\textbf{q}}{|1-tht\textbf{q}|}){(1-{|\textbf{q}|}^{2})}^{\rho-d}dm(\textbf{q})
\end{equation*}

\begin{equation*}
        = {(1-{tht}^{2})}^{s}{\int}_{\{\textbf{q}\in \mathbb{H}, |\textbf{q}|<1\}}{|1-tht\textbf{q}|}^{-2s}{C}_{2l}^{1}(Re(\frac{1-tht\textbf{q}}{|1-tht\textbf{q}|})){(1-{|\textbf{q}|}^{2})}^{\rho-d}dm(\textbf{q})
\end{equation*}

\begin{equation*}
= {(1-{tht}^{2})}^{s} {\int}^{1}_{0} {\int}^{\pi}_{0} {|1-rtht{e}^{i\theta}|}^{-2s}{C}_{2l}^{1}(\frac{1-rtht\cos\theta|}{|1-rtht{e}^{i\theta}|}){(1-{r}^{2})}^{\rho -d}{r}^{3}{{\sin}^{2}\theta} d\theta dr
\end{equation*}
Where we have set:

\begin{equation*}
  \textbf{q}=r(\cos\theta+\sin\theta y); \theta \in [0,\pi];y \in \mathbb{H}, Re(y)=0, \sum {y}_{j}^{2}=1
\end{equation*}
Let $ \eta = rtht $ and set:
\begin{equation*}
{J}_{\lambda,l}(\eta)=   {\int}^{\pi}_{0} {|1-\eta{e}^{i\theta}|}^{-2s}{C}_{2l}^{1} (\frac{1-\eta \cos\theta|}{|1-\eta{e}^{i\theta}|}){\sin}^{2}\theta d\theta
\end{equation*}

Wich may be written $ (\theta\rightarrow \pi-\theta): $
 \begin{equation*}
{J}_{\lambda,l}(\eta)=   {\int}^{\pi}_{0} {|1+\eta{e}^{i\theta}|}^{-2s}{C}_{2l}^{1} (\frac{1+\eta \cos\theta|}{|1+\eta{e}^{i\theta}|}){\sin}^{2}\theta d\theta
\end{equation*}

Set (see [3])
\begin{equation*}
  {e}^{i\omega}= \frac{1+\eta{e}^{i\theta}}{|1+\eta{e}^{i\theta}|}
\end{equation*}

Then we have:

\begin{equation*}
  \cos\omega = \frac{1+ \eta \cos\theta}{|1+\eta{e}^{i\theta}|};\sin\omega = \frac{\eta \sin \theta}{|1+\eta{e}^{i\theta}|}; |\omega|< \pi/2
\end{equation*}

And,

\begin{equation*}
   {C}^{1}_{2l}(\frac{1+ \eta \cos\theta}{|1+\eta{e}^{i\theta}|})= {C}^{1}_{2l}(\cos\omega) = \frac{ \sin (2l+1)\omega }{\sin\omega}
\end{equation*}

\begin{equation*}
   {C}^{1}_{2l}(\frac{1+ \eta \cos\theta}{|1+\eta{e}^{i\theta}|})= \frac{{(1+2\eta\cos\theta+{\eta}^{2})}^{1/2}}{2i\eta \sin\theta}[{e}^{i(2l+1)\omega}-{e}^{-i(2l+1)\omega}]
\end{equation*}

\begin{equation*}
   {C}^{1}_{2l}(\frac{1+ \eta \cos\theta}{|1+\eta{e}^{i\theta}|})= \frac{{(1+2\eta\cos\theta+{\eta}^{2})}^{1/2}}{2i\eta \sin\theta}[{(\frac{1+\eta{e}^{i\theta}}{|1+\eta{e}^{i\theta}|})}^{2l+1}-{(\frac{1+\eta{e}^{-i\theta}}{|1+\eta{e}^{i\theta}|})}^{2l+1}]
\end{equation*}

Then
\begin{equation*}
{J}_{\lambda,l}(\eta)=  \frac{1}{2i\eta} {\int}^{\pi}_{0} {|1+\eta{e}^{i\theta}|}^{-2s-2l}[{(1+\eta{e}^{i\theta})}^{2l+1}-{(1+\eta{e}^{-i\theta})}^{2l+1}]\sin\theta d\theta
\end{equation*}

\begin{equation*}
{J}_{\lambda,l}(\eta)=  \frac{1}{2i\eta} {\int}^{\pi}_{-\pi} {(1+\eta{e}^{i\theta})}^{-s-l}{(1+\eta{e}^{-i\theta})}^{-s-l}{(1+\eta{e}^{i\theta})}^{2l+1}\sin\theta d\theta
\end{equation*}

\begin{equation*}
{J}_{\lambda,l}(\eta)=  \frac{1}{2i\eta} {\int}^{\pi}_{-\pi} {(1+\eta{e}^{i\theta})}^{-s+l+1}{(1+\eta{e}^{-i\theta})}^{-s-l} sin\theta d\theta
\end{equation*}

Hence
\begin{equation*}
{J}_{\lambda,l}(\eta)= \dfrac{\pi}{2}(2l+1){F}_{21}(s-l-1,s+l;2;{\eta}^{2})
\end{equation*}

Using the following lemma (see [3]):

\begin{equation*}
{\int}^{\pi}_{-\pi} \frac{\sin\theta d\theta}{{(1+\eta{e}^{i\theta})}^{\alpha}{(1+\eta{e}^{-i\theta})}^{\beta}}= (\beta-\alpha)\eta \pi i {F}_{21}(\alpha,\beta,2;{\eta}^{2})
\end{equation*}

But we have:
\begin{equation*}
{\Phi}_{\lambda,l}(t)= {(1-{tht}^{2})}^{\frac{i\lambda+\rho}{2}}{\int}_{0}^{1}{(1-{r}^{2})}^{\rho-d}{r}^{3}{J}_{\lambda,l}(rtht)dr
\end{equation*}

Hence\\

$ {\Phi}_{\lambda,l}(t)   = \dfrac{\pi}{2}  (2l+1) {(1-{tht}^{2})}^{\frac{i\lambda+\rho}{2}}{\int}_{0}^{1}{(1-{r}^{2})}^{\rho-d}{r}^{3}\\
{F}_{21}(\frac{i\lambda+\rho}{2}-l-1,\frac{i\lambda+\rho}{2}+l;2;{r}^{2}{tht}^{2})dr $\\

And using Bateman's integral formula for:\\
 $ \mathcal{R}e(c) > \mathcal{R}e(s) > 0, z \neq 1, \mid arg(1-z) \mid <\pi $:\\
 
$ {F}_{21}(a,b,c,z)= \frac{\Gamma(c)}{\Gamma(s)\Gamma(c-s)}{\int}_{0}^{1} {x}^{s-1}{(1-x)}^{c-s-1}{F}_{21}(a,b,s,xz)dx $\\

We can write  $(\rho =2n+1)$:\\

$ {\Phi}_{\lambda,l}(t)= \dfrac{\pi}{4}(2l+1) \frac{\Gamma(2)\Gamma(2n-2)}{\Gamma(2n)}{(1-{tht}^{2})}^{\frac{i\lambda+2n+1}{2}}{F}_{21}(\frac{i\lambda+2n+1}{2}+l,\frac{i\lambda+2n+1}{2}-l-1;2n;{tht}^{2})$

\end{proof}

\section{Precise action of the generalised Poisson transform on ${L}^{2}(\partial B(\mathbb{H}^{n}))$  }

In this section we study the action of the generalised Poisson transform on ${L}^{2}(\partial B(\mathbb{H}^{n}))$.See [5]

Denote by $w_{1},...,w_{n}$ the standard coordinates on ${\mathbb{H}}^{n}$\\

\begin{equation*}
\left\{
                     \begin{array}{ll}
                        w_{1}= \|w\| \cos\xi (\cos\varphi +y \sin \varphi)\\
                      w_{j} = \|w\| {\eta}_{j}\sin \xi
                     \end{array}
                   \right.
\end{equation*}

Where: $0 \leq \xi \leq \pi/2$ ; $0 \leq \varphi \leq \pi$; $y \in \mathbb{H}$\\

And:\begin{equation*}
{|y|}^{2}=1 ; Re(y)=0; {\eta}_{j} \in \mathbb{H};{\sum}_{j=2}^{n}{|{\eta}_{j}|}^{2}=1
\end{equation*}

Let:\begin{equation*}
  {V}_{p,q} = \{ (p,q) \in {\mathbb{N}}^{2}; p \in \mathbb{N}; q-p \in 2\mathbb{N} \}
\end{equation*}

Then we have by the Peter-Weyl theorem:

\begin{equation*}
{L}^{2}(\partial B(\mathbb{H}^{n}))= \sum_{(p,q) \in {V}_{p,q}}{H}_{p,q}
\end{equation*}

And the zonal spherical harmonics that spans ${H}_{p,q}$, first found by Kenneth Johnson and Nolan R. Wallach [5] are given by the following formula (see for instance [10]):

\begin{equation*}
{\phi}_{p,q}(\omega)= \frac{1}{(p+1)} {\cos}^{q}\xi \frac{\sin((p+1)\varphi)}{\sin\varphi} {F}_{21}(\frac{p-q}{2},-\frac{p+q+2}{2};2n-2;-{\tan}^{2}\xi)
\end{equation*}

Now Let's define:

\begin{equation*}
{P}_{\lambda,l}^{r}: {H}_{p,q} \rightarrow {L}^{2}(\partial B(\mathbb{H}^{n}))
\end{equation*}
with
\begin{equation*}
 ({P}_{\lambda,l}^{r}f)(u):= {\int}_{\partial B({\mathbb{H}}^{n})}{P}_{\lambda,l}(ru,v)f(v)dv
\end{equation*}

Then we may state the following proposition:
\begin{prop}
Let $f$ be in ${H}_{p,q}$. Then we have for  $\lambda \in \mathbb{C} $ and $ r \in [0,1[ $:
\begin{equation*}
({P}_{\lambda,l}f)(r.\omega) = {\Phi}_{\lambda,l,p,q}(r)f(\omega)
\end{equation*}

where:
\begin{equation*}
{\Phi}_{\lambda,l,p,q}(r):= ({P}_{\lambda,l}{\phi}_{p,q})(r.{e}_{1})
\end{equation*}

\end{prop}

\begin{proof}
By the $K$-invariance property of the Poisson kernel given in Proposition3.1 we can show that the operator ${P}_{\lambda,l}^{r}$ acts from ${H}_{p,q}$ to ${H}_{p,q}$. \\

Furthermore we have for every K-irreductible representation $\pi_{p,q}$ of ${H}_{p,q}$. :
\begin{equation*}
\pi_{p,q}({k}^{-1}){P}_{\lambda,l}^{r}(f)(u)= ({P}_{\lambda,l}^{r}f)(ku)
\end{equation*}

\begin{equation*}
\pi_{p,q}({k}^{-1}){P}_{\lambda,l}^{r}(f)(u)= {\int}_{\partial B({\mathbb{H}}^{n})}{P}_{\lambda,l}(k.ru,v)f(v)dv
\end{equation*}

\begin{equation*}
\pi_{p,q}({k}^{-1}){P}_{\lambda,l}^{r}(f)(u)= {\int}_{\partial B({\mathbb{H}}^{n})}{P}_{\lambda,l}(ru,{k}^{-1}.v)f(v)dv
\end{equation*}

Setting $\omega= {k}^{-1}.v$ we can write:
\begin{equation*}
\pi_{p,q}({k}^{-1}){P}_{\lambda,l}^{r}(f)(u)= {\int}_{\partial B({\mathbb{H}}^{n})}{P}_{\lambda,l}(ru,\omega)f(k\omega)d\omega
\end{equation*}

Therfore
\begin{equation*}
\pi_{p,q}({k}^{-1}){P}_{\lambda,l}^{r}(f)(u)= {P}_{\lambda,l}^{r}(\pi_{p,q}({k}^{-1}f))(u)
\end{equation*}

Wich gives that we have for each $k \in K $

\begin{equation*}
\pi_{p,q}({k})\circ {P}_{\lambda,l}^{r} = {P}_{\lambda,l}^{r}\circ \pi_{p,q}({k})
\end{equation*}
 
Hence ${P}_{\lambda,l}^{r}$ commutes with all  K-irreductible representations $\pi_{p,q}$ of ${H}_{p,q}$.
 It follows by Schur lemma that the operator ${P}_{\lambda,l}^{r}$ is scalar on each component ${H}_{p,q}$.Hence there exists a constant ${\Phi}_{\lambda,l,p,q}(r)$ such that on each ${H}_{p,q}$ we have :\\
\begin{equation*}
   {P}_{\lambda,l}^{r}= {\Phi}_{\lambda,l,p,q}(r).I 
\end{equation*}

Where $I$ is the identity operator on ${H}_{p,q}$.\\

Taking the zonal spherical function ${\phi}_{p,q}$ we get:
\begin{equation*}
{\Phi}_{\lambda,l,p,q}(r)= ({P}_{\lambda,l}^{r}{\phi}_{p,q})({e}_{1})
\end{equation*}

\end{proof}
 We call the ${\Phi}_{\lambda,l,p,q}$: generalised spherical functions.

\section{Generalised spherical harmonics and asymptotic behaviour }

In this section, we should give the explicit expression of the generalised spherical functions defined below . 

Let's state the following theorem:
\begin{thm}
  The generalised spherical harmonics are given, up to a constant, by:
\begin{equation*}
{\Phi}_{\lambda,l,p,q}(tht)= \frac{\pi}{4. (p+1)}\frac{\Gamma(2)\Gamma(2n-2)}{\Gamma(q+2n)}  {(1-{tht}^{2})}^{s} {(tht)}^{p}
\end{equation*}
\begin{equation*}
[{(s+l)}_{\frac{p+q}{2}+1}{(s-l-1)}_{\frac{q-p}{2}}{F}_{21}(s-l-1+ \frac{q-p}{2}, s+l+1+\frac{p+q}{2}; q+2n;{tht}^{2})
\end{equation*}

\begin{equation*}
- {(s-l-1)}_{\frac{p+q}{2}+1}{(s+l)}_{\frac{q-p}{2}}{F}_{21}(s+l+ \frac{q-p}{2}, s-l+\frac{p+q}{2}; q+2n;{tht}^{2})]
 \end{equation*}

\end{thm}

\begin{proof}
 using again lemma 3.2 we get:\\
 
 $ {\Phi}_{\lambda,l,p,q}(t)= \frac{1}{(p+1)}{(1-{tht}^{2})}^{\frac{i\lambda+\rho}{2}}{\int}_{\{ Z\in \mathbb{H}, |Z|<1\} }{|1-thtZ|}^{-i\lambda-2n-1}\\
 {\chi}_{l}(\frac{1-thtZ}{|1-thtZ|}){(1-{|Z|}^{2})}^{2n-3}h(Z)dm(Z)$\\

Where
\begin{equation*}
h(Z)= {|Z|}^{q}\frac{\sin((p+1)\varphi)}{\sin\varphi} {F}_{21}(\frac{p-q}{2},-\frac{p+q+2}{2};2n-2;\frac{{|Z|}^{2}-1}{{|Z|}^{2}})
\end{equation*}

And
\begin{equation*}
Z= |Z|(\cos\varphi +y \sin \varphi)
\end{equation*}

with
\begin{equation*}
{\sum }_{j\geq 2}{y}_{j}^{2}=1 ; Re(y)=0;
\end{equation*}

Setting $|Z|=r$ we get:\\

 $|1-tht Z|= |1-rtht.{e}^{i\varphi}|$ and $dm(Z)= {r}^{3}{\sin}^{2}\varphi d\varphi dr$\\

we can write:
\begin{center}
\begin{multline*}
 { \Phi}_{\lambda,l,p,q}(t)= \frac{1}{(p+1)} {(1-{tht}^{2})}^{\frac{i\lambda+2n+1}{2}} {\int}^{1}_{0} {\int}^{\pi}_{0} {|1-rtht{e}^{i\varphi}|}^{-i\lambda-2n-1}\\
 {C}_{2l}^{1}(\frac{1-rtht\cos\varphi|}{|1-rtht{e}^{i\varphi}|})
{(1-{r}^{2})}^{ 2n-3}{r}^{q+3}\sin\varphi\sin(p+1)\varphi\\
  {F}_{21}(\frac{p-q}{2},-\frac{p+q+2}{2};2n-2;\frac{{r}^{2}-1}{{r}^{2}})
d\theta dr
\end{multline*}
\end{center}
Which may be written:
\begin{center}
\begin{multline*}
{ \Phi}_{\lambda,l,p,q}(t)= \frac{1}{(p+1)} {(1-{tht}^{2})}^{\frac{i\lambda+2n+1}{2}} {\int}^{1}_{0} [{\int}^{\pi}_{0} {|1-rtht{e}^{i\varphi}|}^{-i\lambda-2n-1}\\
{C}_{2l}^{1}(\frac{1-rtht\cos\varphi|}{|1-rtht{e}^{i\varphi}|}) \sin\varphi\sin(p+1)\varphi d\varphi]  {(1-{r}^{2})}^{ 2n-3}{r}^{q+3}\\
{F}_{21}(\frac{p-q}{2},-\frac{p+q+2}{2};2n-2;\frac{{r}^{2}-1}{{r}^{2}}) dr
\end{multline*}
\end{center}
Set  $\eta:=rtht$ and $s:= \frac{i\lambda+2n+1}{2}$ and define:
\begin{equation*}
I = \frac{1}{(p+1)}{\int}^{\pi}_{0} {|1-\eta{e}^{i\varphi}|}^{-2s}{C}_{2l}^{1}(\frac{1-\eta\cos\varphi}{|1-\eta{e}^{i\varphi}|})\sin\varphi\sin(p+1)\varphi d\varphi
\end{equation*}

Which can be written ($\varphi\mapsto \pi -\varphi$)
\begin{equation*}
I = \frac{1}{(p+1)}{(-1)}^{p} {\int}^{\pi}_{0} {|1+\eta{e}^{-i\varphi}|}^{-2s}{C}_{2l}^{1}(\frac{1+\eta\cos\varphi}{|1+\eta{e}^{i\varphi}|})\sin\varphi\sin(p+1)\varphi d\varphi
\end{equation*}

\begin{equation*}
I = \frac{1}{(p+1)}\frac{{(-1)}^{p}}{2i\eta}2 {\int}^{\pi}_{-\pi} {|1+\eta{e}^{i\varphi}|}^{-2s-2l} {(1+\eta{e}^{i\varphi})}^{2l+1}\sin(p+1)\varphi d\varphi
\end{equation*}

\begin{equation*}
I = \frac{1}{(p+1)}\frac{{(-1)}^{p}}{2i\eta} {\int}^{\pi}_{-\pi} {(1+\eta{e}^{i\varphi})}^{-s+l+1} {(1+\eta{e}^{-i\varphi})}^{-s-l}\sin(p+1)\varphi d\varphi
\end{equation*}

\begin{equation*}
I = \frac{1}{(p+1)} \frac{{(-1)}^{p}}{-4\eta} {\int}^{\pi}_{-\pi} {(1+\eta{e}^{i\varphi})}^{-s+l+1} {(1+\eta{e}^{-i\varphi})}^{-s-l}({e}^{i(p+1)\varphi}-{e}^{-i(p+1)\varphi}) d\varphi
\end{equation*}

\begin{center}
\begin{multline*}
I = \frac{1}{(p+1)}\frac{{(-1)}^{p+1}{(-1)}^{p}}{-4\eta} {\Sigma}_{k,j}\frac{{\eta}^{k+j}{(s-l-1)}_{k}}{k!}\frac{{(s+l)}_{j}}{j!} \\
[{\int}^{\pi}_{-\pi}{e}^{i(k-j+p+1)\varphi}d\varphi -{\int}^{\pi}_{-\pi}{e}^{i(k-j-(p+1))\varphi}d\varphi]
\end{multline*}
\end{center}

\begin{center}
\begin{multline*}
I = \frac{2\pi}{(p+1)}\frac{1}{4\eta} [{\Sigma}_{k=0}^{+\infty}\frac{{\eta}^{2k+p+1}{(s-l-1)}_{k}{(s+l)}_{k+p+1}}{k!(k+p+1)!}\\
-{\Sigma}_{j=0}^{+\infty}\frac{{\eta}^{2j+p+1}{(s-l-1)}_{j+p+1}{(s+l)}_{j}}{j!(j+p+1)!} ]
\end{multline*}
\end{center}

\begin{center}
\begin{multline*}
I = \frac{\pi}{(p+1)}\frac{1}{2\eta}{\eta}^{p+1} [{\Sigma}_{k=0}^{+\infty}\frac{{(s-l-1)}_{k}{(s+l)}_{p+1}{(s+l+p+1)}_{k}}{k!{(2+p)}_{k}{(2)}_{p}}{\eta}^{2k}\\
-{\Sigma}_{j=0}^{+\infty}\frac{{(s-l-1)}_{p+1}{(s+l)}_{j}{(s-l+p)}_{j}}{j!{(2+p)}_{j}{(2)}_{p}}{\eta}^{2k} ]
\end{multline*}
\end{center}

When we have used the formula: ${(a)}_{k}{(a+k)}_{n}= {(a)}_{k+n}$

And
\begin{center}
\begin{multline*}
I = \frac{\pi}{(p+1)}\frac{1}{2}{\eta}^{p} [\frac{{(s+l)}_{p+1}}{{(2)}_{p}}{F}_{21}(s-l-1,s+l+p+1;2+p;{\eta}^{2})\\
-\frac{{(s-l-1)}_{p+1}}{{(2)}_{p}}{F}_{21}(s+l,s-l+p;2+p;{\eta}^{2})]
\end{multline*}
\end{center}

Therefore
\begin{equation*}
{\Phi}_{\lambda,l,p,q}(t)= \frac{\pi}{(p+1)}{(1-{tht}^{2})}^{s}[{I}_{1}-{I}_{2}]
\end{equation*}

With

$ {I}_{1} = \frac{1}{2 \times {(2)}_{p}}{(s+l)}_{p+1}{\int}_{0}^{1}{(1-{r}^{2})}^{2n-3}{r}^{q+3}{(rtht)}^{p}$\\

${F}_{21}(\frac{p-q}{2},-\frac{p+q+2}{2};2n-2;\frac{{r}^{2}-1}{{r}^{2}}){F}_{21}(s-l-1,s+l+p+1;2+p;{(rtht)}^{2})dr$\\

And

$ {I}_{2} = \frac{1}{2 \times {(2)}_{p}}{(s-l-1)}_{p+1}{\int}_{0}^{1}{(1-{r}^{2})}^{2n-3}{r}^{q+3}{(rtht)}^{p}$\\

${F}_{21}(\frac{p-q}{2},-\frac{p+q+2}{2};2n-2;\frac{{r}^{2}-1}{{r}^{2}}){F}_{21}(s+l,s-l+p;2+p;{(rtht)}^{2})dr$\\

Using the classical transformation formula:\\

${F}_{21}(a,b;c;x)= {(1-x)}^{-a}{F}_{21}(a,c-b;c;\frac{x}{x-1})$\\

we can write:

\begin{multline*}
{F}_{21}(\frac{p-q}{2},-\frac{p+q+2}{2};2n-2;\frac{{r}^{2}-1}{{r}^{2}})={r}^{p-q}{F}_{21}(\frac{p-q}{2},\frac{p+q}{2}+2n-1;2n-2;1-{r}^{2})
\end{multline*}

We can easily see that $\frac{q-p}{2}$ is a non negative integer.\\

Then we can write:
\begin{equation*}
{F}_{21}(\frac{p-q}{2},\frac{p+q}{2}+2n-1;2n-2;1-{r}^{2})= \frac{{\mathfrak{J}}_{\frac{q-p}{2}}^{(2n-3,p+1)}(2{r}^{2}-1)}{\binom{\frac{q-p}{2}+2n-3}{\frac{q-p}{2}}}
\end{equation*}
Where we have used the following known formula for $t=1-{r}^{2}$, $\mathfrak{J}$ being a Jacobi polynomial,  :\\

${\mathfrak{J}}_{N}^{(\alpha,\beta)}(1-2{t}^{2})= \binom {N+\alpha}{N}{F}_{21}(-N,N+ \alpha +\beta +1;\alpha +1;{t}^{2})$\\

$ {I}_{1} = \frac{1}{2 \times {(2)}_{p}}{(s+l)}_{p+1}{\int}_{0}^{1}{(1-{r}^{2})}^{2n-3}{r}^{q+3}{(rtht)}^{p}$\\

${F}_{21}(\frac{p-q}{2},-\frac{p+q+2}{2};2n-2;\frac{{r}^{2}-1}{{r}^{2}}){F}_{21}(s-l-1,s+l+p+1;2+p;{(rtht)}^{2})dr$\\

But:

${F}_{21}(\frac{p-q}{2},-\frac{p+q+2}{2};2n-2;\frac{{r}^{2}-1}{{r}^{2}})= \frac{{r}^{p-q} (\frac{q-p}{2})!}{(2n-2)_{\frac{q-p}{2}}} {\mathfrak{J}}_{\frac{q-p}{2}}^{(2n-3,p+1)}(2{r}^{2}-1)$\\

Then we can write:

$ {I}_{1} = \frac{1}{2 \times {(2)}_{p}}{(tht)}^{p}{(s+l)}_{p+1}\frac{(\frac{q-p}{2})!}{(2n-2)_{\frac{q-p}{2}}}{\int}_{0}^{1}{(1-{r}^{2})}^{2n-3}{r}^{2p+3}$\\

${\mathfrak{J}}_{\frac{q-p}{2}}^{(2n-3,p+1)}(2{r}^{2}-1){F}_{21}(s-l-1,s+l+p+1;2+p;{(rtht)}^{2})dr$\\

Setting $ x = 2{r}^{2}-1 $ we get:\\

$ {I}_{1} = \frac{1}{4\times 2 \times {(2)}_{p}}{(tht)}^{p}{(s+l)}_{p+1}\frac{(\frac{q-p}{2})!}{(2n-2)_{\frac{q-p}{2}}}{\int}_{-1}^{1}{(\frac{1-x}{2})}^{2n-3} \times  {{(\frac{1+x}{2})}}^{p+1}$\\

${\mathfrak{J}}_{\frac{q-p}{2}}^{(2n-3,p+1)}(x){F}_{21}(s-l-1,s+l+p+1;2+p;{(\frac{1+x}{2})tht}^{2})dx$\\

$ {I}_{1} = \frac{1}{  {(2)}_{p}}{(tht)}^{p}{(s+l)}_{p+1}\frac{(\frac{q-p}{2})!}{(2n-2)_{\frac{q-p}{2}} \times {2}^{2n+p+1}}{\int}_{-1}^{1}{(1-x)}^{2n-3} \times  {(1+x)}^{p+1}$\\

${\mathfrak{J}}_{\frac{q-p}{2}}^{(2n-3,p+1)}(x){F}_{21}(s-l-1,s+l+p+1;2+p;{(\frac{1+x}{2})tht}^{2})dx$\\

Then using the following known Rodrigues formula for Jacobi polynomials :\\

${\mathfrak{J}}_{N}^{(\alpha,\beta)}(x)= \frac{{(-1)}^{N}}{{2}^{N} N! {(1-x)}^{\alpha}{(1+x)}^{\beta}}  \frac{{d}^{N}}{{dx}^{N}}({(1-x)}^{N+\alpha}{(1+x)}^{N+\beta})$\\

We get:        \\
${\mathfrak{J}}_{{\frac{q-p}{2}}}^{(2n-3,p+1)}(x)= \frac{{(-1)}^{{\frac{q-p}{2}}}}{{2}^{{\frac{q-p}{2}}} ({\frac{q-p}{2}})! {(1-x)}^{2n-3}{(1+x)}^{p+1}}  \frac{{d}^{{\frac{q-p}{2}}}}{{dx}^{{\frac{q-p}{2}}}}({(1-x)}^{{\frac{q-p}{2}}+2n-3}{(1+x)}^{{\frac{q-p}{2}}+p+1})$\\

Wich gives:\\

$ {I}_{1} = \frac{{(-1)}^{\frac{q-p}{2}}}{ {(2)}_{p}}{(tht)}^{p}{(s+l)}_{p+1}\frac{1}{(2n-2)_{\frac{q-p}{2}} \times {2}^{2n+\frac{p+q}{2}+1}}\\
{\int}_{-1}^{1} \frac{{d}^{{\frac{q-p}{2}}}}{{dx}^{{\frac{q-p}{2}}}}({(1-x)}^{{\frac{q-p}{2}}+2n-3}{(1+x)}^{{\frac{p+q}{2}}+1})$ \\
${F}_{21}(s-l-1,s+l+p+1;2+p;{(\frac{1+x}{2})tht}^{2})dx$\\

Then  making an integration by part we can write\\

$ {I}_{1} = \frac{{(-1)}^{\frac{q-p}{2}}}{ {(2)}_{p}}{(tht)}^{p}{(s+l)}_{p+1}\frac{1}{(2n-2)_{\frac{q-p}{2}} \times {2}^{2n+\frac{p+q}{2}+1}} {(-1)}^{\frac{q-p}{2}} ({\frac{{tht}^{2}}{2})}^{\frac{q-p}{2}} \\
  \frac{  {(s-l-1)}_{\frac{q-p}{2}} {(s+l+p+1)}_{\frac{q-p}{2}} }{{(2+p)}_{\frac{q-p}{2}}} {\int}_{-1}^{1} ({(1-x)}^{{\frac{q-p}{2}}+2n-3}{(1+x)}^{{\frac{p+q}{2}}+1})\\
   {F}_{21}(s-l-1+\frac{q-p}{2},s+l+p+1+\frac{q-p}{2};2+p+\frac{q-p}{2};{(\frac{1+x}{2})tht}^{2})dx$\\

$ {I}_{1} = {(-1)}^{q-p} {(tht)}^{q}  {(s+l)}_{\frac{p+q}{2}+1}    {(s-l-1)}_{\frac{q-p}{2}} \frac{1}{{(2)}_{\frac{p+q}{2}} {(2n-2)}_{\frac{q-p}{2}} \times {2}^{2}}   $\\  

$ {\int}_{0}^{1} ({(1-y)}^{{\frac{q-p}{2}}+2n-3}{(y)}^{{\frac{p+q}{2}}+1})\\
{F}_{21}(s-l-1+\frac{q-p}{2},s+l+\frac{p+q}{2}+1;2+\frac{p+q}{2};y{tht}^{2})dy$\\

Recalling that $ q-p \in 2 \mathbb{N} $ and using Bateman's integral formula:
\begin{equation*}
{F}_{21}(a,b,c,z)= \frac{\Gamma(c)}{\Gamma(s)\Gamma(c-s)}{\int}_{0}^{1} {x}^{s-1}{(1-x)}^{c-s-1}{F}_{21}(a,b,s,xz)dx
\end{equation*}

 We can  write :\\
 
$ {I}_{1} =  {(tht)}^{q} \frac{\Gamma(2)\Gamma(2n-2)}{\Gamma(q+2n) \times {2}^{2}}  {(s+l)}_{\frac{p+q}{2}+1}    {(s-l-1)}_{\frac{q-p}{2}}$\\
 
$ {F}_{21}(s-l-1+\frac{q-p}{2},s+l+\frac{p+q}{2}+1;q+2n;{tht}^{2})$\\

And we get similarly :\\

$
 {I}_{2} = {(tht)}^{q} \frac{\Gamma(2)\Gamma(2n-2)}{\Gamma(q+2n) \times {2}^{2}}  {(s-l-1)}_{\frac{p+q}{2}+1}    {(s+l)}_{\frac{q-p}{2}}$\\
 
$ {F}_{21}(s+l+\frac{q-p}{2},s-l+\frac{p+q}{2};q+2n;{tht}^{2})$\\
\end{proof}

The main tool in the proof of the main theorem is the determination of the asymptotic behaviour of the generalized spherical functions.This result arises as direct consequence of the following theorem:[2]

\begin{thm}
For $(a-b), \alpha, \beta \in \mathbb{N} $ and $ a,b,c \in \mathbb{C}, Re(a+b+\alpha+\beta-c-1)> 0 $ we have:\\
\begin{equation*}
{lim}_{Z\rightarrow1}(1-Z)^{-c+a+b+\alpha+\beta-1}( {(a)}_{\alpha}{(b)}_{\beta}{F}_{21}(a+\alpha,b+\beta;c;z)-{(a)}_{\beta}{(b)}_{\alpha}{F}_{21}(a+\beta,b+\alpha;c;z)) \end{equation*}
\begin{equation*}
=\frac{\Gamma(c)}{\Gamma(a)\Gamma(b)}\Gamma(a+b+\alpha+\beta-c-1)(a-b)(\alpha-\beta)
\end{equation*}
\end{thm}

Taking $ a:= s+l, b:=s-l-1, \alpha:= \frac{p+q}{2}+1,\beta:= \frac{q-p}{2} $and $ c:=q+ 2n $ we can easily conclude that:

\begin{equation*}
{lim}_{r\rightarrow 1^{-}}{(1-{r}^{2})}^{-(\frac{2n+1-i\lambda}{2})}{\Phi}_{\lambda,l,p,q}(r) = {C}_{l}(\lambda)
\end{equation*}
 where:
\begin{equation*}
{C}_{l}(\lambda) = \frac{\pi}{4}(2l+1)\frac{\Gamma(2n-2)\Gamma(2s-2n-1)}{ \Gamma(s+l)\Gamma(s-l-1)}
\end{equation*}

Then we may state the following corollary:

\begin{cor}
Let $\lambda \in \mathbb{C} \backslash \mathbb{R}$ with $ Re(i\lambda)>0 $ , then there exist a constant ${C}_{l}(\lambda)$ such as :
\begin{equation*}
{lim}_{r\rightarrow 1^{-}}{(1-{r}^{2})}^{-(\frac{2n+1-i\lambda}{2})}{\Phi}_{\lambda,l,p,q}(r) = {C}_{l}(\lambda)
\end{equation*}
 where:
\begin{equation*}
{C}_{l}(\lambda) = \frac{\pi}{4}(2l+1)\Gamma(2n-2)\frac{\Gamma(i\lambda)}{ \Gamma(\frac{i\lambda+2n+1}{2}+l)\Gamma(\frac{i\lambda+2n+1}{2}-l-1)}
\end{equation*}
\end{cor}

\begin{rem}
Using the following propriety of contiguous hypergeometric functions:
 \begin{equation*}
a {F}_{21}(a+1,b,c,z) - b {F}_{21}(a,b+1,c,z) = (b-a) {F}_{21}(a,b,c,z)
\end{equation*}

   we find,for the case $l=0$,(up to a constant) the same expressions of the known generalised spherical functions. See for instance [8],[11].
   
   Morever, for $ p = q = 0$, we get, (up to a constant), the same expression of the elementary spherical function given by Proposition3.3.

\end{rem}

\section{Proof of the main theorem }
In this section we give the proof of the main theorem.\\
But let's first state the following proposition:\\

\begin{prop}
Let $ \lambda \in \mathbb{C} $ with $ Re(i\lambda)>0 $, then there exists a positive constant $ \delta_{l}(\lambda) $ such that for $ f \in {L}^{\textbf{p}}(\partial B({\mathbb{H}}^{n}), \textbf{p}\geq 2 $ we have:

\begin{equation*}
sup_{0\leq r < 1}{(1-{r}^{2})}^{-\frac{2n+1-Re(i\lambda)}{2}}{[{\int}_{\partial B({\mathbb{H}}^{n})}{|({P}_{\lambda,l}f)(r\omega)|}^{\textbf{p}}d\omega]}^{1/\textbf{p}}<{\delta}_{l}(\lambda){\|f\|}_{{L}^{p}(\partial B({\mathbb{H}}^{n}))}<+\infty
\end{equation*}
 \end{prop}

\begin{proof}
Let $   F = {P}_{\lambda,l}(f),f \in {L}^{\textbf{p}}(\partial B({\mathbb{H}}^{n}),\textbf{p}\geq 2, \lambda \in \mathbb{C} $ with $ Re(i\lambda)>0. $ \\

 Then we have (setting $ \theta = k e_{1}, k \in K ) $ :

\begin{equation*}
{\int}_{\partial B({\mathbb{H}}^{n})}{P}_{\lambda,l}(r\omega,\theta)f(\theta)d\theta = {\int}_{K}{P}_{\lambda,l}(r\omega,k{e}_{1})f(k{e}_{1})dk
\end{equation*}

\begin{equation*}
{\int}_{\partial B({\mathbb{H}}^{n})}{P}_{\lambda,l}(r\omega,\theta)f(\theta)d\theta = {\int}_{K}{P}_{\lambda,l}(r.h{e}_{1},k{e}_{1})f(k{e}_{1})dk, \omega= h.{e}_{1}, h \in K
\end{equation*}

\begin{equation*}
{\int}_{\partial B({\mathbb{H}}^{n})}{P}_{\lambda,l}(r\omega,\theta)f(\theta)d\theta = {\int}_{K}{P}_{\lambda,l}(r{e}_{1},{h}^{-1}k{e}_{1})f(k{e}_{1})dk
\end{equation*}

let's introduce the function $ P_{\lambda,r} $ on $ K $ as follows:\\

$ {P}_{\lambda,r}(k)={P}_{\lambda}(r.{e}_{1},k{e}_{1}) $
 Then we have:
\begin{equation*}
({P}_{\lambda,l}f)(r\omega) = ({P}_{\lambda,r}\ast f)(k{e}_{1})
\end{equation*}

Using the Hausdorff-Young inequality we get:
\begin{equation*}
{\|{P}_{\lambda,r}\ast f\|}_{{L}^{p}(\partial B({\mathbb{H}}^{n}))}\leq {\|{P}_{\lambda,r}\|}_{{L}^{1}(\partial B({\mathbb{H}}^{n}))}{\|f\|}_{{L}^{p}(\partial B({\mathbb{H}}^{n}))}
\end{equation*}

Morever we have,
\begin{equation*}
{\|{P}_{\lambda,r}\|}_{{L}^{1}(\partial B({\mathbb{H}}^{n})}= {\int}_{K}|{P}_{\lambda,r}(k{e}_{1})|dk 
\end{equation*}

But

${\int}_{K}|{P}_{\lambda,r}(k{e}_{1})|dk = {\int}_{\partial B({\mathbb{H}}^{n})}|{P}_{\lambda,l}(r{e}_{1},\omega)|d\omega  $
Then 
\begin{equation*}
{\int}_{K}|{P}_{\lambda,r}(k{e}_{1})|dk = |{\Phi}_{i\lambda,l}(r)|; 
\end{equation*}

Then we can write:
\begin{equation*}
{\|{P}_{\lambda,r}\|}_{{L}^{1}(\partial B({\mathbb{H}}^{n})}= {\Phi}_{Re(i\lambda),l}(r)
\end{equation*}
And:
\begin{multline*}
{\|{P}_{\lambda,r}\|}_{{L}^{1}(\partial B({\mathbb{H}}^{n})}= \frac{\pi}{4}(2l+1){(1-{r}^{2})}^{\frac{Re(i\lambda)+2n+1}{2}}\\
{F}_{21}(\frac{Re(i\lambda)+2n+1}{2}+l,\frac{Re(i\lambda)+2n+1}{2}-l-1;2n;{r}^{2})
\end{multline*}

We conclude that for $Re(i\lambda)>0$ we have:

\begin{equation*}
sup_{0\leq r < 1}{(1-{r}^{2})}^{-\frac{2n+1-Re(i\lambda)}{2}}{[{\int}_{\partial B({\mathbb{H}}^{n})}{|({P}_{\lambda,l}f)(r\omega)|}^{\textbf{p}}d\omega]}^{1/\textbf{p}}\leq {\delta}_{l}(\lambda){\|f\|}_{{L}^{\textbf{p}}(\partial B({\mathbb{H}}^{n}))}<+\infty
\end{equation*}

Where: \begin{equation*}
{\delta}_{l}(\lambda):= \frac{\pi}{4}(2l+1)\frac{\Gamma(2n)\Gamma(Re(i\lambda))}{\Gamma(\frac{2n+1+Re(i\lambda)}{2}-l-1)\Gamma(\frac{2n+1+Re(i\lambda)}{2}+l)}
\end{equation*}
Now we turn to the proof of Theorem 2.3\\

Suppose that $ \textbf{p} =2 $ and Let \begin{equation*}
 F = {P}_{\lambda,l}(f),f \in {L}^{2}(\partial B({\mathbb{H}}^{n}).
\end{equation*}
Expanding $f$ into its $K$-type serie $f = {\sum}_{p,q\in \Omega }{{f}_{p,q}}$ we get:

\begin{equation*}
 F(ru) = {\sum}_{p,q \in \Omega} {\Phi}_{\lambda,l,p,q}(r){f}_{p,q}(u) in {\emph{C}}^{\infty}([0,1[\times \partial B({\mathbb{H}}^{n}) )
\end{equation*} 

Morever we have: ${\sum}_{p,q \in \Omega} {\mid{\Phi}_{\lambda,l,p,q}(r)\mid}^{2} { \parallel{{f}_{p,q}}\parallel}^{2}_{2} < \infty$ for all $ r \in [0,1[$\\

Then for every fixed $ r \in [0,1[ $, the series $ {\sum}_{p,q \in \Omega} {\Phi}_{\lambda,l,p,q}(r){f}_{p,q}(v) $ are uniformly convergents on $\partial B({\mathbb{H}}^{n}) $.\\

Lets now consider for each $ r \in [0,1[$ the following $\mathbb{C}$ valued function $ {g}_{r} $ on the boundary $ \partial B({\mathbb{H}}^{n}) ) $ given by:
\begin{equation*}
{g}_{r}(u)= {(1-{r}^{2})}^{-(2n+1-Re(i\lambda))}\int_{\partial B(\mathbb{H}^{n})}F(rv)  \overline{{P}_{\lambda,l}(ru,v)}dv
\end{equation*}
Then
\begin{equation*}
{g}_{r}(u)= {(1-{r}^{2})}^{-(2n+1-Re(i\lambda))}\int_{\partial B(\mathbb{H}^{n})}{\sum}_{p,q\in \Omega} {\Phi}_{\lambda,l,p,q}(r){f}_{p,q}(v)\overline{{P}_{\lambda,l}(ru,v)}dv
\end{equation*}
Since the series $ {\sum}_{p,q \in \Omega} {\Phi}_{\lambda,l,p,q}(r){f}_{p,q}(v) $ are uniformly convergents on $\partial B({\mathbb{H}}^{n}) )$ we get:\\
\begin{equation*}
{g}_{r}(u)= {(1-{r}^{2})}^{-(2n+1-Re(i\lambda))}{\sum}_{p,q\in \Omega }{\Phi}_{\lambda,l,p,q}(r)\int_{\partial B(\mathbb{H}^{n})}\overline{{P}_{\lambda,l}(ru,v)}{f}_{p,q}(v)dv
\end{equation*}

Hence:
\begin{equation*}
{g}_{r}(u)= {(1-{r}^{2})}^{-(2n+1-Re(i\lambda))}{\sum}_{p,q\in \Omega }{\mid{\Phi}_{\lambda,l,p,q}(r)\mid}^{2}{f}_{p,q}(u)
\end{equation*}
Then:

\begin{equation*}
{\Vert {{C}_{l}(\lambda)}^{-2} {g}_{r}-f \Vert}_{2}^{2}= {\sum}_{p,q\in \Omega}{[{{C}_{l}(\lambda)}^{-2}{(1-{r}^{2})}^{-(2n+1-Re(i\lambda))}{\mid{\Phi}_{\lambda,l,p,q}(r)\mid}^{2}-1]}^{2}{\Vert}{f}_{p,q}\Vert^{2}  
\end{equation*}
Then we conclude from the asymptotic behaviour of the generalized spherical functions $ {\Phi}_{\lambda,l,p,q}(r) $ that:
$ {lim}_{r\longrightarrow{1}^{-}}  {{\Vert {C}_{l}(\lambda))}^{-2} {g}_{r}-f \Vert}_{2}^{2} = 0 $ which gives the desired result.

\end{proof}

Now we turn to the proof of Theorem 2.2 for the case \textbf{p}=2;

\begin{proof}
The necessary condition is given by proposition 6.1.\\

For the sufficient condition let's now suppose that $F= {P}_{\lambda,l}f$ satisfies the Hardy's inequality:
\begin{equation*}
sup{(1-{r}^{2})}^{-\frac{2n+1-Re(i\lambda)}{2}}{[{\int}_{\partial B({\mathbb{H}}^{n})}{|({P}_{\lambda,l}f)(r\omega)|}^{2}d\omega]}^{1/2}<+\infty
\end{equation*}
Then we have to show that $f\in {{L}^{2}(\partial B({\mathbb{H}}^{n}))}$\\

Decomposing $ f $ to it's $ K $ type series $ f = {\sum}_{p,q \in \Omega}{f}_{p,q} $we can write in $ \mathcal{C}^{\infty}(\left[0,1 \right] \times \partial B({\mathbb{H}}^{n}) $

\begin{equation*}
{P}_{\lambda,l}f (ru)= \Sigma_{p,q \in \Omega} {\Phi}_{\lambda,l,p,q}(r){f}_{p,q}(u)
\end{equation*}\\
Morever we have for every fixed $ r \in \left[0,1 \right[ $ : 

\begin{equation*}
{(1-{r}^{2})}^{-\frac{2n+1-Re(i\lambda)}{2}}({\sum}_{p,q\in \Omega}{|{\Phi}_{\lambda,l,p,q}(r)|}^{2} {\|{f}_{p,q}\|}^{2}_{2})^{1/2}<+\infty
\end{equation*}

If we consider $\Omega_{0}$ un arbitrary finite subset of $\Omega$ then we have for every $ r \in \left[0,1 \right[ $:\\

\begin{equation*}
{(1-{r}^{2})}^{-\frac{2n+1-Re(i\lambda)}{2}}({\sum}_{p,q\in \Omega_{0}}{|{\Phi}_{\lambda,l,p,q}(r)|}^{2} {\|{f}_{p,q}\|}^{2}_{2})^{1/2}\leq {\|{P}_{\lambda,l}f\|}_{\lambda,\ast,2} <+\infty
\end{equation*}

And from the corollary giving the asymptotic behaviour of ${\Phi}_{\lambda,l,p,q}(r)$ we conclude:

\begin{align*}
{{C}_{l}(\lambda)}^{2}{\sum}_{p,q\in \Omega_{0} }{\|{f}_{p,q}\|}^{2}_{2} \leq sup{(1-{r}^{2})}^{-\frac{2n+1-Re(i\lambda)}{2}}{[{\int}_{\partial B({\mathbb{H}}^{n})}{|({P}_{\lambda,l}f)(r\omega)|}^{2}d\omega]}^{1/2}<+\infty
\end{align*}

\end{proof}

Now we are ready to give the proof of the main theorem 2.1 for all   $\textbf{p} \geq 2 $

\begin{proof}
The necessary condition follows from Proposition6.1\\

For the sufficient condition, let's consider $F$ a $\mathbb{C}$-valued function on $\partial B({\mathbb{H}}^{n})$ such that $ {\Vert F \Vert}_{\lambda,\textbf{p},\ast}< +\infty$.\\

Using the fact that ${\Vert F \Vert}_{\lambda,\textbf{p},\ast}< {\Vert F \Vert}_{\lambda,2,\ast}  $ then we conclude from the Theorem 2.2 that there exist a function $ f \in {L}^{2}(B({\mathbb{H}}^{n}))$ such that:\\
 $ F = {P}_{\lambda,l}(f) $.
 Furthermore, by Theorem 2.3 we have $ f(u) = {lim}_{r\longrightarrow{1}^{-}}{g}_{r}(u) $ where:
\begin{equation*}
 {g}_{r}(u)= {(C_{l}(\lambda))}^{-2} {(1-{r}^{2})}^{-(2n+1-Re(i\lambda))}\int_{\partial B(\mathbb{H}^{n})}F(rv)  \overline{{P}_{\lambda,l}(ru,v)}dv
 \end{equation*}

For every continuous function in $ \partial B(\mathbb{H}^{n}) $ we can write:
\begin{equation*}
{lim}_{r\longrightarrow{1}^{-}}{\int}_{\partial B({\mathbb{H}}^{n})}{g}_{r}(u) \overline{\Phi(u)}du = {\int}_{\partial B({\mathbb{H}}^{n})}f(u) \overline{\Phi(u)}du
\end{equation*}
Morever we have:
\begin{center}
\begin{multline*}
{\int}_{\partial B({\mathbb{H}}^{n})}{g}_{r}(u) \overline{\Phi(u)}du = {C_{l}(\lambda)}^{-2} {(1-{r}^{2})}^{-(2n+1-Re(i\lambda))}\\
\int_{\partial B(\mathbb{H}^{n}}(\int_{\partial B(\mathbb{H}^{n}}F(rv) \overline{{P}_{\lambda,l}(ru,v)}dv)\overline{\Phi(u)}du 
\end{multline*}
\end{center}
Then,

\begin{equation*}
{\int}_{\partial B({\mathbb{H}}^{n})}{g}_{r}(u) \overline{\Phi(u)}du = {C_{l}(\lambda)}^{-2} {(1-{r}^{2})}^
{-(2n+1-Re(i\lambda))}\int_{}\int_{\partial B(\mathbb{H}^{n}}  \overline{{P}_{\lambda,l}\Phi(rv)}F(rv)dv
\end{equation*}

And by using the Holder inequalty for $ \textbf{p} $ being such that $ 1/\textbf{p}+1/\textbf{q} = 1 $ we obtain:
\begin{center}
\begin{multline*}
|{\int}_{\partial B({\mathbb{H}}^{n})}{g}_{r}(u) \overline{\Phi(u)}du| \leq {C_{l}(\lambda)}^{-2}  {(1-{r}^{2})}^{-(2n+1-Re(i\lambda))}\\
{(\int_{\partial B(\mathbb{H}^{n}} { {|({P}_{\lambda,l}\Phi)(rv)|}^{\textbf{q}}}dv)}^{1/\textbf{q}}{\Vert F \Vert}_{\lambda,\textbf{p},\ast}
\end{multline*}
\end{center}
 
Then:

\begin{equation*}
|{\int}_{\partial B({\mathbb{H}}^{n})}{g}_{r}(u) \overline{\Phi(u)}du| \leq {C_{l}(\lambda)}^{-2}  {(1-{r}^{2})}^{-(2n+1-Re(i\lambda))} {\Vert({P}_{\lambda,l}\Phi)}\Vert _{\textbf{q}}{\Vert F \Vert}_{\lambda,\textbf{p},\ast}
\end{equation*}

But we have from proposition 4.1 and corollary 5.3: \\

$ \Phi(u) = C_{l}(\lambda)^{-1} lim_{r \rightarrow 1^{-}}{(1-{r}^{2})}^{(2n+1-(i\lambda)/2)}({P}_{\lambda,l}\Phi)(ru)  $  in  $ {L}^{\textbf{q}}(B({\mathbb{H}}^{n})) $\\

Hence, 
\begin{equation*}
{\Vert f \Vert}_{p} \leq {(C_{l}(\lambda))}^{-1}{\Vert \Phi\Vert}_{q}{\Vert F \Vert}_{\lambda,p,\ast}
\end{equation*}
We finally  conclude:
\begin{equation*}
C_{l}(\lambda)\Vert f \Vert_{p}  \leq {\Vert F \Vert}_{\lambda,p,\ast}
\end{equation*}

\end{proof}

\subsection*{Acknowledgment}
Many thanks to the work group in our mathematics department for fruitful discussion during the development of this work.

\end{document}